\documentclass[final]{scrartcl}

\usepackage{amsfonts,amssymb,amsmath}
\usepackage{showkeys}
\usepackage{graphicx,epsfig}
\usepackage{algorithm,algorithmic}
\usepackage[latin1]{inputenc}
\usepackage[normalem]{ulem}
\usepackage{verbatim}

\newcommand{\C}{\mathbb{C}}

\newcommand{\la}{\lambda}
\newcommand{\g}{\gamma}
\newcommand{\Cg}{\mathcal{C}_\gamma}

\renewcommand{\le}{\leqslant}
\renewcommand{\ge}{\geqslant}
\renewcommand{\leq}{\leqslant}
\renewcommand{\geq}{\geqslant}

\newcommand{\wh}{\widehat}
\newcommand{\wt}{\widetilde}

\newcommand{\cs}{\mathcal{V}}
\newcommand{\lcs}{\mathcal{U}}
\newcommand{\ncs}{\bar{\mathcal{V}}}
\newcommand{\sss}{\mathcal{W}}

\DeclareMathOperator{\diag}{diag}
\DeclareMathOperator{\sspan}{span}
\DeclareMathOperator{\gap}{gap}
\DeclareMathOperator{\sep}{sep}
\DeclareMathOperator{\relsep}{relsep}
\DeclareMathOperator{\dist}{dist}

\newcommand{\vett}[2]{\begin{bmatrix}#1\\ #2\end{bmatrix}}
\newcommand{\rowvett}[2]{\begin{bmatrix}#1 & #2\end{bmatrix}}
\newcommand{\twotwo}[4]{\left[\begin{array}{cc} #1 & #2\\ #3 & #4 \end{array}\right]}

\newcommand{\norm}[1]{\left\Vert#1\right\Vert}
\newcommand{\abs}[1]{\left\vert#1\right\vert}

\newtheorem{theorem}{Theorem}
\newtheorem{lemma}[theorem]{Lemma}
\newtheorem{Test}{Test}
\newenvironment{proof}{\textit{Proof. }}{\hfill $\square$ \bigskip}

\renewcommand{\H}{\ensuremath{\mathcal{H}}}

\usepackage{hyperref}

\begin{document}


\title{A Subspace Shift Technique for Nonsymmetric Algebraic Riccati Equations}

\author{Bruno Iannazzo\footnote{Dipartimento di Matematica e Informatica. Via Vanvitelli 1, 06123 Perugia, Italy \href{mailto:bruno.iannazzo@dmi.unipg.it}{\texttt{bruno.iannazzo@dmi.unipg.it}}. The work of the first author was partly supported by PRIN 2008 N. 20083KLJEZ and by GNCS of Istituto Nazionale di Alta Matematica.} \and Federico Poloni\footnote{Technische Universit\"at Berlin, Strasse des 17 Juni 137, 10623 Berlin \href{mailto:poloni@math.tu-berlin.de}{\texttt{poloni@math.tu-berlin.de}}. The work of the second author was supported by a postdoctoral grant of the A.~von Humboldt Foundation since May 2011.}}

\maketitle

\begin{abstract}
The worst situation in computing the minimal nonnegative solution
of a nonsymmetric algebraic Riccati equation 
associated with an M-matrix occurs when the corresponding linearizing matrix has two very small eigenvalues, one with positive and
one with negative real part. When both these eigenvalues
are exactly zero, the problem is called critical or null recurrent. While in this case the problem is ill-conditioned
and the convergence of the algorithms based on matrix iterations is slow,
there exist some techniques to
remove the singularity and transform the problem to a well-behaved one.
Ill-conditioning and slow convergence appear also in
close-to-critical problems, but when none of the eigenvalues is exactly zero
the techniques used for the critical case cannot be applied.

In this paper, we introduce a new method to accelerate the convergence properties of the
iterations also in close-to-critical cases, by working on the
invariant subspace associated with the problematic eigenvalues as a whole.
We present a theoretical analysis and several
numerical experiments which confirm the efficiency of the new method.
\end{abstract}


\section{Introduction}\label{sec:intro}

We consider the nonsymmetric algebraic Riccati equation (or
NARE) 
\begin{equation}\label{eq:NARE}
   XCX-AX-XD+B=0,
\end{equation}
where $X,B\in\mathbb C^{m\times n}$, $A\in\mathbb C^{m\times m}$,
$C\in\mathbb C^{n\times m}$, $D\in\mathbb C^{n\times n}$. We write equation \eqref{eq:NARE} briefly as $\mathcal R(X)=0$ where $\mathcal R(X)=XCX-AX-XD+B$.

In certain applications in queueing models \cite{rogers} and in
the numerical solution of transport equations \cite{juanglin}, the
coefficients of \eqref{eq:NARE} are such that
\[
 \mathcal{M}=\begin{bmatrix}
    D & -C\\
    -B & A
   \end{bmatrix}
\]
is an M-matrix, either nonsingular or singular irreducible. In
this case, we give Equation \eqref{eq:NARE} the acronym M-NARE. We
recall that $M\in\C^{n\times n}$ is an M-matrix if it can be
written in the form $M=sI_n-N$, where $I_n$ is the identity matrix
of size $n$ (denoted also by $I$ if there is no ambiguity), $N$ is
a matrix whose elements are nonnegative, for which we use the
notation $N\ge 0$, and $s\ge \rho(N)$, where $\rho(\cdot)$ is the
spectral radius of a square matrix. The M-matrix $M$ is singular
if $s=\rho(N)$ and nonsingular if $s>\rho(N)$. It can be proved
that the eigenvalues of an M-matrix have nonnegative real part
\cite{berple}.

The solutions of the NARE \eqref{eq:NARE} can be put in
correspondence with certain $n$-dimensional invariant subspaces of
the matrix
\begin{equation}\label{defH}
 \mathcal{H}=\begin{bmatrix}
    D & -C\\
    B & -A
   \end{bmatrix}.
\end{equation}
More precisely, a matrix $X\in \C^{m\times n}$ is a solution of
\eqref{eq:NARE} if and only if the columns of $\begin{bmatrix}
  I_n \\ X
 \end{bmatrix}$
span an invariant subspace of $\mathcal{H}$. In particular, it holds
that
\begin{equation}\label{eq:HIX}
\mathcal{H}\begin{bmatrix}
  I_n \\ X
 \end{bmatrix}=\begin{bmatrix}
  I_n \\ X
 \end{bmatrix}(D-CX),
\end{equation}
and the eigenvalues of $D-CX$ are a subset of the eigenvalues
of~$\mathcal{H}$. 

We say that the NARE \eqref{eq:NARE} is {\em associated with} the
matrix $\mathcal{H}$ of \eqref{defH} or that $\H$ is the {\em linearizing} matrix of the NARE. Observe that any $2\times 2$
block matrix with square diagonal blocks yields a NARE associated
with it.

In the case of an M-NARE it can be proved
that the eigenvalues of
$\mathcal{H}$ can be ordered by non increasing real part such that
\begin{equation}\label{eq:ordeig}
    \Re\la_1\geq \cdots\geq \Re\la_{n-1}\geq \la_n\geq 0\geq \la_{n+1}\geq \cdots\geq
    \Re\la_{m+n},
\end{equation}
that is, $n$ eigenvalues belong to the closed right half complex
plane and the other eigenvalues to the closed left half plane, and
the eigenvalues
$\la_n$ and $\la_{n+1}$ are real. If moreover $\mathcal M$ is irreducible, then $\Re\la_{n-1}>\la_n\geq 0\geq \la_{n+1}>\Re \la_{n+2}$.
If $\la_n=0=\la_{n+1}$,
then the eigenvalue zero is
associated to a size-2 Jordan block  (see
\cite{bimp} and the references therein).

All these spectral properties implies that the matrix $\mathcal{H}$ associated with an
M-NARE has a unique $n$-dimensional invariant subspace
corresponding to the $n$ {\em rightmost} eigenvalues, namely
$\la_1,\ldots,\la_n$, which we call the $n$-dimensional {\em
antistable invariant subspace} of $\mathcal{H}$ (the term comes
from the theory of the symmetric algebraic Riccati equations in
dynamical systems \cite{lr}). On the other hand, the matrix
$\mathcal{H}$ has a unique $m$-dimensional invariant subspace
corresponding to the $m$ {\em leftmost} eigenvalues, namely
$\la_{n+1},\ldots,\la_{n+m}$, which we call {\em stable invariant
subspace}.

In the applications, the required solution of the M-NARE is the
one for which the columns of $\begin{bmatrix} I_n\\X
\end{bmatrix}$ span the $n$-dimensional antistable
invariant subspace of $\mathcal{H}$, or, equivalently, such that
the eigenvalues of $D-CX$ are the $n$ rightmost eigenvalues of
$\mathcal{H}$. This solution has been proved to exist and it turns
out to be the minimal element-wise nonnegative solution of
\eqref{eq:NARE} (see \cite{guo01}).

Equation \eqref{eq:NARE} is usually solved either by some matrix
iteration, e.g., the Cyclic Reduction (CR) \cite{bm96} or the
Structured Doubling Algorithm (SDA) \cite{cfl05,gim},
whose limits yield the required solution or using the ordered
Schur form of $\mathcal{H}$ \cite{guo06}.

Both the conditioning of the equation and the convergence speed of most iterations are strictly related to some property of good separation between the stable and anti-stable subspace; for this purpose, several different measures of ``nearness'' are used in literature; in Section~\ref{sec:gapsep} we introduce and discuss them. Nevertheless, all approaches identify two important cases:
\begin{enumerate}
\item $\la_n=\la_{n+1}=0$; in this case, the minimal nonnegative solution of~\eqref{eq:NARE} is ill-conditioned~\cite{guohigham} and the convergence of most iterations degrades from quadratic to linear~\cite{seisamurai}. This case is known as \emph{critical case}.
Most of this problems can be circumvented by using the so-called \emph{shift technique}~\cite{gim,hmr01}. It consists in making a special rank-one correction of $\mathcal{H}$, obtaining a new Riccati equation with the same minimal solution. The new equation has better conditioning and the convergence of iterations is quadratic again. The shift technique works not only in the critical case, but also when only one of $\la_n$ and $\la_{n+1}$ is zero, yielding minor benefits in this case.
\item equations that are (in some sense) ``close'' to having $\la_n=\la_{n+1}=0$, while neither of the two eigenvalues is exactly zero. This case is known as \emph{close to critical}. It is effectively the worst-case scenario, since the same difficulties as in the critical case appear (ill-conditioning, slow convergence of the numerical methods based on matrix iterations), but the shift technique cannot be applied as it requires an eigenvalue to be exactly zero. 
\end{enumerate}
The difficulties in the latter case are the main motivation for this work. We present a new technique, which we call \emph{subspace shift}, that aims to extend the shift technique to this case. A necessary assumption is that we can identify a small subset of the eigenvalues which are ``responsible'' for the ill-conditioning of the equation and well separated from the rest of the spectrum; we define this property more precisely in the following. We call the associated subspace \emph{central subspace}. The dimension $k$ of this subspace may be known \emph{a priori} from the theoretical properties of the problem (as for instance in~\cite{juanglin}), or determined at run-time, which is a more difficult task.

The technique consists in building a rank-$k$ modification of the matrix that alters the eigenvalues associated with the central subspace, but does not modify the invariant subspaces and the minimal solution. In this way, we reduce the problem to a ``far from critical'' one, for which the solution can be computed with a faster and stabler iterative method.

The paper is organized as follows. In Section~\ref{sec:gapsep}, we describe and compare the different notions of distance from criticality which exist in literature and how they affect the conditioning of the problem and the convergence speed of the numerical algorithms. In Section~\ref{sec:ds} we introduce the subspace shift technique and outline some results that give an insight of its behavior in terms of the different criticality metrics. In Section~\ref{sec:alg} we describe its implementation in more detail and discuss the computational aspects. Section~\ref{sec:numa} contains some experimental results that show the effectiveness of this technique. Finally, Section \ref{sec:conc} draws the conclusions.

In the following, $\sigma(M)$ stands for the set of the
eigenvalues of $M\in \C^{n\times n}$, and $\norm{\cdot}_F$ denotes
the Frobenius norm. We define the \emph{Cayley transform} of
parameter $\gamma\in \mathbb R\setminus\{0\}$ as the map
\[
 \Cg: z \mapsto \frac{z-\gamma}{z+\gamma}.
\]
Notice that, for $\gamma>0$, $\Cg$ maps the open (closed) right
half-plane onto the open (closed) unit circle, and the open
(closed) left half-plane onto the exterior of the open (closed)
unit circle.


\section{Measures of criticality: $\gap$ and $\sep$}\label{sec:gapsep}
\subsection{The gap between eigenvalues}
The simplest measure of criticality, adopted in most works on the shift \cite{hmr01,gim} is the so-called \emph{gap}, i.e., the distance between the two eigenvalues closer to the imaginary axis
$
 \gap(\mathcal{H}):=\abs{\la_{n}-\la_{n+1}}.
$
In the critical case $\gap(\mathcal{H})=0$, and a problem is called close-to-critical when $\gap(\mathcal{H})$ is small with respect to $\norm{\H}$. A strictly related quantity, which appears explicitly in the expressions for the convergence speeds of SDA and Cyclic Reduction \cite{bimp}, is its Cayley-transformed version
\begin{equation}\label{convspeed}
 \gap_{\Cg}(\mathcal{H}):=\frac{\max_{i=1,\dots,n} \abs{\Cg(\lambda_i)}}{\min_{j=1,\dots,m} \abs{\Cg(\lambda_{n+j})}},
\end{equation}
where $\gamma$ is chosen according to
\begin{equation}\label{eq:gamma1}
\gamma\ge \gamma_*=\max\left\{\max_{1\le i\le m}a_{ii}, \max_{1\le
i\le n}d_{ii}\right\}.
\end{equation}
We have $\gap_{\Cg}(\mathcal{H})\leq 1$, with equality in the critical case. Since the Cayley transform alters the position of the eigenvalues, it is not apparent that the minimum and maximum in \eqref{convspeed} are attained in $\la_{n}$ and $\la_{n+1}$; we give here a proof of this result. Let us first
assess the following technical lemma.
\begin{lemma}\label{lemma}
Let $\Gamma$ be a closed disc in the complex plane with center
$C\in\mathbb R$ and radius $r$. The point in $\Gamma$ with maximal
modulus is one among $C+r$ and $C-r$.
\end{lemma}
\begin{proof}
Let $C+p\in \mathbb C$, $\abs{p}\leq r$, be a generic point in the
disc. By the triangle inequality, $\abs{C+p}\leq \abs{C}+\abs{p}
\leq \abs{C}+r$, with equality if and only if $\abs{p}=r$ and $p$
has the same argument as $C$, i.e., either real positive or
negative.
\end{proof}
\begin{theorem}
Let $\mathcal{H}$ be associated with an M-NARE, and $\gamma$ be
chosen according to \eqref{eq:gamma1}. The minimum and the maximum
in \eqref{convspeed} are attained by $i=n$ and $j=1$, i.e., we may replace \eqref{convspeed}
with
\[
 \gap_{\Cg}(\mathcal{H}):=\frac{\abs{\Cg(\lambda_n)}}{\abs{\Cg(\lambda_{n+1})}}.
\]
\end{theorem}\label{thm:central}
\begin{proof} From
\eqref{eq:gamma1} we have $\gamma I-D\geq 0$ and thus $P=\gamma I
- D +CX_\ast \geq 0$. Hence we may write $D-CX_\ast=\gamma I - P$;
from the Perron--Frobenius theory of M-matrices, it follows that
all the eigenvalues of $D-CX_\ast$ are contained in the closed
disc with center $\gamma$ and radius $r=\gamma-\lambda_n$; and in
particular, the eigenvalue $\lambda_n$ lies on its boundary. We
call this disc $\Gamma$, and proceed to prove that
$\Cg(\lambda_n)$ has the maximal modulus among all points in in
$\Cg(\Gamma)$. The image of $\Gamma$ under the Cayley transform is
a closed disc $\Gamma'$, which must be contained in the unit disc
and symmetric with respect to the real axis. This means that its
center (which is not in general $\Cg(\gamma)$) is real. This disc
$\Gamma'$ intersects the real axis in the two points
$\Cg(\lambda_n)$ and $\Cg(2\gamma-\lambda_n)$. By the lemma, the
point of maximal modulus in $\Gamma'$ is one among them; direct
computation (using $\lambda_n\leq \gamma$) shows that it is the
former.

A similar reasoning starting from $A-X_\ast C$ yields that
$\min_{j=1,\dots,m} \abs{\Cg(\lambda_{n+j})}$ is achieved by
$j=1$; we need some extra care with the signs, as
$\lambda_{n+1}\leq 0$, and with the fact that this time the image
of the enclosing disc under the Cayley transform is the
\emph{outside} of a suitable disc.
\end{proof}

Therefore an explicit relation among the two concepts of gap can be established. Observe that while the gap changes by scaling the matrix $\H$ by a real parameter $\alpha\ne 0$, the Cayley-transformed gap of $\alpha \H$ is the same as the one of $\H$, if the same value of $\gamma$ is chosen according to \eqref{eq:gamma1}. Thus, the Cayley transformed gap can be seen as a relative inverse gap. 


\subsection{The subspace separation}\label{sec:subsep}
A third, more accurate notion of nearness is given by the \emph{subspace separation} \cite{golubvanloan,stewartsun}. We first define the \emph{separation} between the two square matrices $M$ and $N$ as
\begin{equation}\label{eq:defsep}
 \sep(M,N):=\min_{X\neq 0} \frac{\norm{MX-XN}}{\norm{X}},
\end{equation}
where $\norm{\cdot}$ is a suitable matrix norm (for instance we denote by $\sep_F$ and $\sep_2$ the separation in the Frobenius and spectral norm, respectively), and recall the bound
\begin{equation}\label{gapsep}
 \sep(M,N) \leq \min_{\mu \in \sigma(M),\nu \in
 \sigma(N)}\abs{\mu-\nu}.
\end{equation}
Given an invariant subspace $\mathcal{W}$ for a matrix $A\in\mathbb{C}^{(n+m)\times(n+m)}$, let
$Q$
be an unitary matrix such that
\begin{align}\label{defA11A12}
 Q^* A Q =&
 \twotwo{A_{11}}{A_{12}}{0}{A_{22}},&  A_{11}\in& \C^{n\times n}, & A_{22}\in& \C^{m\times m},
\end{align}
and the first $n$ columns of $Q$ span $\mathcal{W}$. We define $\sep(\mathcal{W}):=\sep(A_{11},A_{22})$
and $\relsep(\mathcal{W}):=\frac{\sep(\mathcal{W})}{\norm{A}}$.

When $A=\mathcal{H}$ and $\mathcal{W}$ is the anti-stable space,
$\relsep(\mathcal{W})$ gives a third measure of the distance of $\mathcal{H}$ from the critical case. The conditioning of the Riccati equation depends essentially on this separation measure, as shown by the following results. Let the
\emph{distance} between two subspaces be defined as
$\dist(\mathcal{U},\mathcal{V}):=\norm{P_{\mathcal{U}}-P_{\mathcal{V}}}$,
with $P_{\mathcal{W}}$ the orthogonal projection on the image of
$\mathcal{W}$.
\begin{theorem}[\cite{stewartsun}]
Let $\mathcal{W}$ be an invariant subspace of a matrix
$\mathcal{H}$ and $\widetilde{\mathcal{W}}$ be an invariant
subspace of $\widetilde{\mathcal{H}}=\mathcal{H}+E$, where
$\norm{E}\leq \varepsilon \norm{\mathcal{H}}$. Then, for all sufficiently small $\varepsilon$ we have
\begin{equation}\label{relsepbound}
 \dist(\mathcal{W},\widetilde{\mathcal{W}})\leq C\frac{\varepsilon}{\relsep{\mathcal{W}}},
\end{equation}
for a suitable constant $C$ of moderate size.
\end{theorem}
The result is stated in a stronger form in \cite{stewartsun},
using the norms of $E_{12}$ and
$\mathcal{H}_{12}$ in two suitable block partitions of the
involved matrices, but here we favor this form for the sake of simplicity.

A bound on the subspace distance can be transformed
into a bound on the solutions of the associated Riccati equations.
\begin{lemma}\label{svdtrick}
 Let $\mathcal{W}=\sspan{\begin{bmatrix}I_n\\X\end{bmatrix}}$, $\widetilde{\mathcal{W}}=\sspan{\begin{bmatrix}I_n\\\widetilde{X}\end{bmatrix}}$. Then, for both spectral and Frobenius norm, we have
$
 \norm{X-\widetilde{X}}\leq \left(\norm{I_n}^2+\norm{X}^2\right)^{1/2}\left(\norm{I_n}^2+\norm{\widetilde{X}}^2\right)^{1/2}\dist(\mathcal{W},\widetilde{\mathcal{W}}).
$
\end{lemma}
\begin{proof}
We use the characterization of $\dist$ given in
\cite[Theorem~2.6.1]{golubvanloan}; the proof there refers to the
spectral norm, but it can be adapted to the Frobenius norm. We
apply the result to the orthonormal matrices
\begin{align*}
 W_2=&\begin{bmatrix}-X^*\\I\end{bmatrix}(I+XX^*)^{-1/2}, & Z_1=\begin{bmatrix}I\\\widetilde{X}\end{bmatrix}(I+\widetilde{X}^*\widetilde{X})^{-1/2},
\end{align*}
then use norm submultiplicativity
\[ \norm{X-\widetilde{X}}\!\leq\!\norm{(I+X^*X)^{1/2}}\!\norm{(I+X^*X)^{-1/2}[-X\ \ I]\vett{I}{\widetilde{X}}(I+\widetilde{X}^*\widetilde{X})^{-1/2}}\!\norm{(I+\widetilde{X}^*\widetilde{X})^{1/2}}.
\]
Finally, the terms $\norm{(I+X^*X)^{-1/2}}$ and $\norm{(I+\wt X^*\wt X)^{-1/2}}$ can be transformed into the desired form by taking an SVD of $X$ and $\wt X$, respectively.
\end{proof}

A tighter bound, which still behaves essentially as $O((\relsep{\mathcal{W}})^{-1})$, can be found in \cite{XXGuoBai}.


\subsection{Relations between $\gap$ and $\sep$}

If the Frobenius norm is used, the $\gap(\mathcal H)$ and the $\sep(\mathcal W)$ are the smallest eigenvalue and the smallest singular value of the matrix
$(I\otimes A_{11}-A_{22}^T\otimes I)$,
respectively (compare \eqref{eq:defsep} and \eqref{defA11A12}). It is well known that the two numbers coincide for normal matrices~\cite[Exercise 5.2.1]{stewartsun}, but may differ significantly in the nonnormal case \cite[Example~5.2.4]{stewartsun}. The same happens for any norm, and in view of \eqref{gapsep} we obtain that $\sep(\mathcal{W})\leq \gap(\mathcal{H})$.

For nonnormal matrices the subspace separation is a better tool to gauge the distance from criticality, especially when conditioning properties are in exam. However, as far as we know, all the literature regarding shift methods deals only with the gap as a measure of criticality. The geometrical intuition is clearer in the gap setting and the proofs are easier to carry on. On the other hand, the whole point of shifting strategies is getting rid of the two real eigenvalues $\la_n$ and $\la_{n+1}$ close to the origin, and it is less clear how we should define the gap in when the two eigenvalues have been moved. 
In view of \eqref{gapsep}, we extend the definition of gap in the following way. Let $\mathcal{W}$ be an invariant subspace of a matrix $A$, and $A_{11},A_{22}$ as in \eqref{defA11A12}. We set
\[
 \gap(\mathcal{W}):=\min_{\lambda\in\sigma(A_{11}),\mu\in\sigma(A_{22})} \abs{\lambda-\mu}.
\]
Similarly, we define the gap between two invariant subspaces of $A$ as
\[
 \gap(\mathcal{U},\mathcal{V}):=\min_{\text{$\lambda$ associated to $\mathcal{U}$, $\mu$ associated to $\mathcal{V}$}}\abs{\lambda-\mu}.
\]
In the following, we define our subspace shift technique in terms of the gap metric, because only by resorting to eigenvalue location criteria we can select suitable subspaces for its application. When discussing conditioning, moving to the $\sep$ setting (and having good separation properties in this setting) is necessary. However, as far as we know, even a complete theory of the basic shift technique in terms of $\sep$ does not exist at present; the intuitive assertion that ``things get better when we move the eigenvalues more far apart'' is difficult to formalize in terms of the $\sep$ metric. We are not able to give full proof for many of the conditioning-related assertions, but we provide at least partial ones that show our claims when the separation bounds behave as suggested by the gap metric analogy. This in particular includes the case in which $\H$ is normal or departs only slightly from normality.


\section{Theoretical bases}\label{sec:ds}
\subsection{The shift technique}\label{shift}

The shift technique has been applied in \cite{gim} to the M-NARE
\eqref{eq:NARE} where $\mathcal{M}$ is a singular irreducible
M-matrix, that is, when at least one between $\la_n$ and
$\la_{n+1}$ is $0$. Without loss of generality one can assume that
$\la_n=0$: the case $\la_n>0=\la_{n+1}$ can be reduced to the case
$\la_n=0$ by a simple trick \cite[Lemma 5.1]{gim}.

The shift technique is rooted in the following results.
\begin{lemma}[Brauer's theorem \cite{brauer52}]\label{th:brauer}
Let $(\lambda, v)$ be an eigenpair for the matrix $T$. Let $u$ be
a vector with $u^*v=1$ and $s$ be a scalar. The eigenvalues of the
matrix $\wh{T} := T+svu^*$ are the same as those of $T$, except for one occurrence of
$\lambda$ which is replaced by $\lambda+s$.
\end{lemma}
\begin{theorem}[\cite{gim}]\label{th:gim}
Let $\mathcal{H}$ be the as in \eqref{defH} associated with the
M-NARE \eqref{eq:NARE} with $\lambda_n=0$, and let $v_n$ be an
eigenvector relative to $\lambda_n$; consider the matrix
\[
\wh{\mathcal{H}}:= \mathcal{H} + sv_nu^*,
\]
with $u^*v_n=1$ and $s>0$. Then, the minimal solution $X_*$ of the
M-NARE associated with $\H$ is a solution of the NARE associated
with $\wh{\H}$. Moreover, $m$ eigenvalues of $\wh{\H}$ lie in the closed
left half plane and $n$ in the open right half plane, whose
corresponding invariant subspace is spanned by the the columns of
$\vett I{X_*}$.
\end{theorem}
The shift technique consists in computing one eigenvector $v_n$
corresponding to the eigenvalue $\lambda_n=0$, and using it to
construct the NARE associated with $\wh{\mathcal{H}}$, which we
call the {\em shifted NARE}. The matrix $\wh{\mathcal{H}}$ has
eigenvalues $\la_1, \ldots, \la_{n-1},
\wh{\la}_n,\la_{n+1},\ldots,\la_{n+m}$ where $\wh{\la}_n=s$ (the
eigenvalue $\la_n$ has been ``shifted'' from $0$ to $s$, this
justifies the name of the technique). Observe that $\gap(\wh{\mathcal{H}})>\gap(\mathcal{H})$;
thus, better conditioning and faster convergence are
expected, once the Cayley parameter $\gamma$ is fixed. It has been
proved in \cite{gim} that SDA applied to the shifted equation,
using the same Cayley parameter as the nonshifted case, but with
the initial values as in \eqref{SDAini}, constructed from $\wh{\H}$,
converges quadratically with a better rate of convergence than the
nonshifted case. Numerical experiments \cite{gim,bilm} show that
this technique reduces dramatically the number of steps of
iterations like SDA and CR. In the critical case, the convergence from linear becomes
quadratic.

It can happen that the Riccati equation associated with
$\wh{\mathcal{H}}$ is not an M-NARE; that is, $\wh{\mathcal
M}=\left[\begin{smallmatrix} I & 0\\0& -I\end{smallmatrix}\right] \wh{\mathcal{H}}$
need not be an M-matrix. Hence,
there is no guarantee that the SDA can be carried out without
breakdown, even if in practice this method works well and the
applicability of SDA is usually assumed \cite{gim}.

Since $\lambda_n=0$, the vector $v_n$ can be computed easily
as $\ker \mathcal{M}$. In principle, the shift technique could be
used also for nonsingular M-matrices, i.e., the hypothesis
$\lambda_n=0$ is not actually needed in Theorem~\ref{th:gim}.
In this case there is no simple relation among the eigenvectors of $\mathcal{M}$ and $\mathcal{H}$, and different techniques are needed for the computation of $\lambda_n$
and $v_n$; for instance, the power method. However, in the close-to-critical case the eigenvector $v_n$ is ill-conditioned and therefore it cannot be computed with good
accuracy.

To solve this problem, we present in Section \ref{sec:sushitech} a new technique which shifts a whole invariant subspace containing $\la_n$ and $\la_{n+1}$,
without attempting to separate the eigenvectors. To this purpose we use an invariant subspace whose associated eigenvalues are well separated from the rest of the spectrum.


\subsection{Results on separation}
We first provide a couple of simple lemmas that will be used in the following.
\begin{lemma}\label{thm:lemJ}
Let the spectrum of $M\in\C^{n\times n}$ be the union of two
disjoint sets, say $\Lambda_1$ and $\Lambda_2$. Let the columns of
$U$ and $V$ (with full column rank) span the left and right invariant subspaces of $M$
corresponding to the eigenvalues in $\Lambda_1$, respectively.
Then $U^*V$ is nonsingular.

Moreover, let the columns of $W$ (with full column rank) span a right invariant subspace
of $M$ whose corresponding eigenvalues belong all to $\Lambda_2$, then $U^*W=0$.
\end{lemma}
\begin{proof}Let $L^{-1}ML=J$ be the Jordan canonical form of
$M$ where the Jordan blocks are ordered such that the $k$ eigenvalues
in $\Lambda_1$ come first, then, a basis of the right invariant subspace
of $M$ corresponding to $\Lambda_1$ is made by the first $k$ columns of
$L$, say $L_1$. Thus $V=L_1P$ for some nonsingular
$P\in\C^{k\times k}$. Similarly, $U^*=QR_1^*$ where $R_1^*$ are
the first $k$ rows of $L^{-1}$ and $Q\in\C^{k\times k}$ is
nonsingular, thus $U^*V=QR_1^*L_1P=QP$ is nonsingular. By a
similar argument, it can be proved that left and right invariant
subspaces corresponding to different eigenvalues are orthogonal.
\end{proof}
\begin{lemma}\label{miscellany}
 Let
\[
 A=\twotwo{A_{11}}{A_{12}}{0}{A_{22}},
\]
where $\sigma(A_{11}) \cap \sigma(A_{22})=\emptyset$.
\begin{enumerate}
\item \label{diagz} There is a matrix
\[
 Z_e=\twotwo{I}{Z}{0}{I}
\]
with $\norm{Z}\leq\frac{\norm{A_{12}}}{\sep(A_{11},A_{22})}$ such that $Z_e A Z_e^{-1}=\diag(A_{11},A_{22})$
\item \label{ownfact} For each $B$, $\sep(A_{11},B)\geq \sep(A,B)$ and $\sep(A_{22},B)\geq \sep(A,B)$.
\item \label{sepcond} $\sep(M,N)\leq \sep(T_1MT_1^{-1},T_2NT_2^{-1})\kappa(T_1)\kappa(T_2)$.
\item \label{diagsep} If $M=\diag(M_1,M_2)$ and $N=\diag(N_1,N_2)$, then $\sep(M,N)=\min_{i,j=1,2}\sep(M_i,N_j)$.
\end{enumerate}
\end{lemma}
Items \ref{diagz}, \ref{sepcond} and \ref{diagsep} are in \cite{stewartsun}. Item \ref{ownfact} follows from the definition of $\sep$, by noting that the minimum of $AX-XB$ increases if we restrict to matrices $X$, partitioned conformably with $A$ as $X=\rowvett{X_1}{X_2}^*$, where one of the two blocks is zero.


\subsection{The subspace shift technique}\label{sec:sushitech}
Let $\cs$ and $\ncs$ be two complementary invariant subspaces of $\mathcal{H}$ satisfying the following conditions:
\begin{enumerate}
 \item $\cs$ has small dimension $k$.
 \item $\cs$ is well separated from $\ncs$, i.e., $\gap(\cs,\ncs)\geq \delta_1$, for a $\delta_1>0$ not excessively small.
 \item $\ncs$ does not contain eigenvalues close to the imaginary axis, i.e., $\gap(\ncs_s,\ncs_u)\geq\delta_2$, for a $\delta_2>0$ not excessively small, where $\ncs_s$ and $\ncs_u$ are the invariant subspaces associated with the stable and anti-stable part of $\ncs$.
\end{enumerate}
Let $V, U\in\C^{(n+m)\times k}$ be matrices whose
orthonormal columns span respectively $\cs$ and the left invariant subspace $\lcs$ associated with the same eigenvalues as $\cs$. Notice that $\lcs$ is well defined, since the subset of $\sigma(\H)$ associated to $\cs$ and $\ncs$ are disjoint because $\gap(\cs,\ncs)>0$. By Lemma~\ref{thm:lemJ}, $U^*V$ is nonsingular.

We consider the matrix
\begin{equation}\label{htilde2}
\wh{\H}=\H(I+sV(U^*V)^{-1}U^*),
\end{equation}
which is a rank $k$ modification of $\H$. Its spectral properties are summarized by the
following result.
\begin{theorem}\label{thm:shiftk}
Let the spectrum of $\H\in\C^{(n+m)\times (n+m)}$ be the union of two disjoint sets, say $\Lambda_1=\{\xi_1,\ldots,\xi_k\}$ and $\Lambda_2=\{\xi_{k+1},\ldots,\xi_{n+m}\}$.
Let $V, U\in\C^{(n+m)\times k}$ be matrices whose
orthonormal columns span the right and left invariant subspaces associated with the eigenvalues of $\Lambda_1$, respectively. The matrix $\wh{\H}$ in \eqref{htilde2} has
the same right invariant subspaces as $\H$ and eigenvalues
$\{(1+s)\xi_1, \ldots, (1+s)\xi_k, \xi_{k+1}, \ldots,
\xi_{n+m}\}$.
\end{theorem}
\begin{proof}
As above, let $\cs$ be the right invariant subspace of $\H$ corresponding to the eigenvalues in $\Lambda_1$ and let $\ncs$ be the right invariant subspace complementary to $\cs$, corresponding to the remaining eigenvalues $\{\xi_{k+1},\dots,\xi_{m+n}\}$. Let $\mathcal{W}_1\subset \cs$ be a right invariant
subspace of $\H$ spanned by the column of the matrix
$W_1\in\C^{(n+m)\times \ell_1}$, then $W_1=VQ$ for some matrix
$Q\in\C^{k\times \ell_1}$, thus
\[
    \wh{\H} W_1=\H (W_1 + s V(U^*V)^{-1}U^*VQ)=\H (W_1+sVQ)=\H
    W_1(1+s),
\]
and thus $\{(1+s)\xi_1,\ldots,(1+s)\xi_k\}$ are eigenvalues of $\wh
\H$.

Let $\mathcal{W}_2\subset \ncs$ be a right invariant
subspace of $\H$ spanned by the column of the matrix
$W_2\in\C^{(n+m)\times \ell_2}$, then by Lemma \ref{thm:lemJ} we have $U^*W_2=0$ and hence $\wh{\H}
W_2=\H W_2$. Thus, $\{\xi_{k+1},\ldots,\xi_{m+n}\}$ are
eigenvalues of $\wh{\H}$.
Since any right invariant subspace of $\H$ can be written as the
sum of two invariant subspaces contained in $\cs$ and
$\ncs$, respectively, the proof is completed.
\end{proof}

Theorem \ref{thm:shiftk} can be applied to the linearizing matrix $\H$ of a close-to-critical M-NARE \eqref{eq:NARE}, where $\mathcal V$ is chosen to be a subspace containing both $\la_n$ and $\la_{n+1}$. Then the anti-stable invariant subspace $\mathcal{W}$ and thus the Riccati solution $X_*$ of the NARE associated to $\wh{\H}$ are the same as the ones for \eqref{eq:NARE}.

From the point of view of the eigenvalue location and of the gap metric, the behavior of the subspace shift technique is clear: the eigenvalues closer to the imaginary axis, which are responsible for the slow convergence and ill-conditioning, are multiplied by a factor $1+s$, which takes them farther from the imaginary axis and thus improve the gap. If the factor $1+s$ is not too large then the Cayley gap is reduced and the numerical algorithms based on matrix iterations converge faster. In order to give a complete treatment of the stability properties of the technique, we have to resort to the separation metric.


\subsection{Conditioning of $U^*V$}
For the subspace shift technique, we need to
form $(U^*V)^{-1}$; therefore, it is crucial that the condition of this matrix is not worse than the conditioning of the problem we are solving. Similarly to what happens in the original problem, we can relate its conditioning to the separation between $\cs$ and $\ncs$.
\begin{theorem}\label{conduv}
Let $U$ and $V$ be orthonormal bases for $\lcs$ and $\cs$ as defined above. Then, we have
\[
 \norm{(U^*V)^{-1}}\leq C\relsep(\mathcal{V})^{-1}+D
\]
for moderate constants $C, D>0$.
\end{theorem}
\begin{proof}
Perform an orthonormal change of basis so that
$V=\begin{bmatrix}I\\0\end{bmatrix}$ and partition
\[
 \mathcal{H}=\begin{bmatrix}A_{11} & A_{12} \\ 0 & A_{22}\end{bmatrix}.
\]
Let $Z$ as in item~\ref{diagz} of Theorem~\ref{miscellany}; in particular we have $\norm{Z}\leq \relsep(\cs)^{-1}$.
Notice that $(I+ZZ^*)^{-1/2}\begin{bmatrix}I & -Z\end{bmatrix}$ is
another orthonormal basis of $\lcs$,
thus $U^*=Q(I+ZZ^*)^{-1/2}\begin{bmatrix}I & -Z\end{bmatrix}$ for
a suitable unitary $Q$. We have then
\[
\norm{(U^*V)^{-1}}=\norm{(I+ZZ^*)^{1/2}}=(\norm{I}^2+\norm{Z}^2)^{1/2},
\]
where the last equality can be established as in the proof of Lemma~\ref{svdtrick}.
\end{proof}


\subsection{Conditioning of the shifted problem}
While not a formal proof, the following argument can be used to gauge the conditioning of the shifted problem. Apply an orthogonal change of basis to bring $\H$ in the form
\[
 \begin{bmatrix}
  V_a & G_a & * & *\\
  & \bar{V}_a & * & *\\
 & & V_s & G_s\\
 & & & \bar{V}_s
 \end{bmatrix},
\]
where $\sigma(V_s)$ and $\sigma(V_a)$ contain respectively the stable and antistable eigenvalues associated with $\cs$, and $\sigma(\bar{V}_s)$, $\sigma(\bar{V}_a)$ the stable and antistable eigenvalues associated with $\ncs$. We may find a matrix $R_a$ such that
\[
 R_a\twotwo{V_a}{G_a}{0}{\bar{V}_a}R_a^{-1}=\diag(V_a,\bar{V}_a),
\]
and $\kappa(R_a)=O\Bigl(\frac{\norm{G_a}^2}{\sep(V_a,\bar{V}_a)^2}\Bigr)$. If $\gap(V_a,\bar{V_a})$ is a good approximation of the corresponding $\sep$, or if $\cs$ and $\ncs$ are well separated also in the $\sep$ sense, we expect this condition numbers to be moderate. We argue similarly for the blocks indexed by $s$, and construct a matrix $R_s$ which annihilates $G_s$. We let $R:=\diag(R_a,R_s)$.

Using the special block structure of $R\H R^{-1}$, we may construct the matrices $U$ and $V$ needed for the shift technique as
\begin{align*}
 V=&R^{-1}
\begin{bmatrix}
I & *\\
0 & *\\
0 & I\\
0 & 0
\end{bmatrix},&
U=&R
\begin{bmatrix}
 I & 0\\
0 & 0\\
* & I\\
* & 0
\end{bmatrix},
\end{align*}
with $U^*V=I$. Therefore $\wh{\H}$ takes the form
\[
 \wh{\H}=R^{-1}\begin{bmatrix}
          (s+1)V_a & 0 & * & *\\
           0 & \bar{V}_a & * & *\\
           0 & 0 & (s+1)V_s & 0\\
           0 & 0 & 0 & \bar{V}_s
         \end{bmatrix}R,
\]
where the entries marked with an asterisk might be affine functions of $s$. Thanks to Lemma~\ref{miscellany}, the separation of the anti-stable subspace in $\wh{\H}$ is
\begin{equation}\label{quibariamo}
 \wh{\sep\mathcal{\sss}} \geq \frac{\min\left(\sep((s+1)V_a,(s+1)V_s),\sep((s+1)V_a,\bar{V}_s),\sep(\bar{V}_a,(s+1)V_s),\sep(\bar{V}_a,\bar{V}_s)\right)}{\kappa(R_a)\kappa(R_s)}.
\end{equation}
If $s$ is moderate and good separation properties hold among $\cs$ and $\ncs$, and between the stable and unstable part of $\cs$, then we may expect, in analogy with the gap setting, that this minimum is attained by the first element, i.e., $\wh{\sep\sss}\geq (s+1)\frac{1}{\kappa(R_a)\kappa(R_s)}\sep \sss$. This means that, up to factors that depends on the separation of the central subspace $\sep{\cs}$, the conditioning of the shifted problem is improved by a factor $s+1$.

Unfortunately, a full proof of the fact that the minimum in \eqref{quibariamo} is attained by the critical subspaces seems elusive. As far as we know, there are no results in literature concerning the behavior of the separation when a scaling is applied to one of the two matrices in a way that takes the eigenvalues ``more far apart'' (in some suitable setting). The analogy with the gap setting and the experimental results in Section~\ref{sec:numa} seem to back up our claim.


\subsection{Stability under perturbations of the computed central subspace}
With a similar reasoning, we may try to estimate the impact of errors in the computed bases for the left and right central subspaces on the accuracy of the solution. We choose a basis as in Theorem~\ref{conduv}. If we use perturbed
versions of $U^*$ and $V$ as
\begin{align*}
 V=&\begin{bmatrix}I+E_1\\E_2\end{bmatrix} & U^*=\begin{bmatrix}I+F_1 & -Z+F_2\end{bmatrix},
\end{align*}
we obtain instead of $\wh{\H}$
\[
 \widetilde{{\mathcal{H}}}
=
\left[\begin{smallmatrix}(s+1)H_{11}+s(E_1H_{11}+H_{11}F_1+E_{1}H_{11}F_{1}) &
H_{12}-sH_{11}Z+s(-E_1H_{11}Z+H_{11}F_2+E_1H_{11}F_2)
\\ sE_2H_{11}(I+F_1) &H_{22}+sE_2H_{11}(-Z+F_2)\end{smallmatrix}\right].
\]
If $\norm{E_i},\norm{F_i}\leq \varepsilon$, then
$\norm{\widehat{\mathcal{H}}-\widetilde{{\mathcal{H}}}}= K s\varepsilon\norm{Z}\norm{\H}$ for a moderate constant $K$.
From the reasoning in the previous section, we expect $\wh{\relsep\sss}^{-1}\leq \frac{{\kappa(R_a)\kappa(R_s)}}{s+1}\relsep\sss^{-1}$, thus by \eqref{relsepbound} the computed subspace $\widetilde{\sss}$ satisfies
\[
 \dist(\sss,\widetilde{\sss})\leq C\frac{Ks\varepsilon\norm{Z}}{\widehat{\relsep{\sss}}}=\frac{s}{s+1} \frac{CK\kappa(R_a)\kappa(R_s)\norm{Z}\varepsilon}{\relsep{\sss}}.
\]
The factor $\frac{s}{s+1}$ is bounded by 1, and
this bounds differs from the perturbation bound \eqref{relsepbound}
for the nonshifted problem only by factors depending on $\relsep{\cs}$.


\section{The SuShi (Subspace Shift) algorithm}\label{sec:alg}

We assume that the space $\cs$  corresponding to the smallest in modulus eigenvalues of $\H$ verifies the assumptions stated at the begin of Section \ref{sec:sushitech}. We call $\cs$ the {\em central subspace} of $\H$.

In this case, the subspace shift technique of Section \ref{sec:sushitech} can be
easily translated into a numerical algorithm for solving close-to-critical
M-NARE \eqref{eq:NARE}. 
We call it SuShi
(Subspace Shift) algorithm.

\begin{algorithm}
 \caption{SuShi algorithm for the
solution of a close-to-critical M-NARE}\label{algo:ss1}
\begin{algorithmic}[1]
 \STATE choose $k$
 \STATE compute two matrices $U,V$ with orthogonal columns which span the left and right invariant
 subspaces of $\mathcal{H}$ corresponding to its $k$ smallest in modulus eigenvalues,  $\xi_1,\ldots,\xi_k$,
 respectively
 \STATE choose $s>0$ and compute $\wh{\mathcal{H}}=\mathcal{H}(I+sV(U^*V)^{-1}U^*)$
 \STATE solve the NARE $\wh{\mathcal{R}}(X)=0$ associated with $\wh{\mathcal
 H}$, to get the minimal nonnegative solution $X_*$ of the original M-NARE $\mathcal R(X)=0$
\end{algorithmic}
\end{algorithm}

Algorithm 1 is very simple. Nevertheless, in order to
get a decent implementation, the details need to be tackled with
some care. One could ask how to dynamically determine the value of
$k$, how to compute $U$ and $V$, how to determine the value of $s$
and how to efficiently solve the ``shifted'' NARE $\wh {\mathcal R}(X)=0$.
This is the topic of the next sections.


\subsection{Computation of the central invariant subspaces}

The central left and right invariant subspaces are the ones corresponding to the smallest in modulus eigenvalues of $\H$, then the inverse orthogonal iteration of \cite[Section 9.3.2]{golubvanloan} applied to $\H$ and $\H^*$, respectively, converges to these subspaces. Notice that we assume that $\H$ is nonsingular so that the customary shift technique cannot be used.

The luckiest situation arises when $k=2$ and
the eigenvalues $\la_n$ and $\la_{n+1}$ of $\mathcal{H}$ are the smallest in modulus and well
separated from the others, as in the problem of \cite{juanglin}.

In the general case, setting $k=2$ may not yield the desired results since we could have close-to-critical settings
such that $\lambda_n$ and $\lambda_{n-1}$ are very close to each other
and to zero, and $\lambda_{n+1}$ slightly larger in modulus than both of them. Moreover, even if we shift away $\la_n$ and $\la_{n+1}$, the
remaining eigenvalues of $\H$ (e.g., $\la_{n-1}$ in a setting similar to the case
above) could be very close to them, and thus the  convergence speed of the Riccati solvers based on matrix iterations is almost unchanged; therefore, the subspace shift with $k=2$ can still be applied but is not much useful.

For these reasons, we would like to handle cases where $k$ may be larger than two, this means that there is some
eigenvalue different from $\la_n$ and $\la_{n+1}$ near to the
origin.
Our minimal assumption is that there exists a value of $k\ge 2$ such that the $k$ smallest eigenvalues of  $\H$ are near to each other (close-to-critical case) and well separated from the $(2n-k)$ largest eigenvalues (compare the assumptions of Section \ref{sec:sushitech}). 

If the value of $k$ is not known in advance it can be determined using the following strategy:
start the inverse
orthogonal iteration  with $k=2$, estimate its convergence speed after some steps. If the iteration shows
itself too slow, enlarge $k$ until the convergence becomes (possibly) fast. This approach yields together $V$ and $k$. The matrix $U$ can be
obtained applying a subspace iteration to $\mathcal{H}^*$. 


\subsection{Magnitude of the shift parameter}\label{sec:selshift}

Another issue which appears in the practical implementation is the
selection of the shift parameter $s$ in Algorithm 1, which should be automatic to deal with different problems. 

A major concern is that if the chosen value of $s$
is too small, then the two central eigenvalues do not move
significantly and the gap remains small; on the other hand, if the
shift parameter is excessively large then $\norm{\widehat{\mathcal{H}}}_F$
grows, and the conditioning of the shifted Riccati equation
degrades, according to the results of \ref{sec:gapsep}.

Let $\xi_1,\ldots,\xi_{m+n}$ be the eigenvalues of $\H$ ordered by nondecreasing modulus.
When the main objective is the acceleration of matrix iterations, like the SDA, the value of $s$ should be chosen such that the eigenvalues corresponding to the central subspaces, say $\xi_1,\ldots,\xi_k$, which are responsible of the slow convergence, gets far from the imaginary axis. 

We need to estimate how small they are with respect to the other eigenvalues.
We may get an estimate using the convergence speed of the inverse orthogonal iteration. Notice that, with our
assumptions on the eigenvalues, the convergence rate is determined by
\[
 t=\frac{|\xi_k|}{|\xi_{k+1}|},
\]
hence a rough estimate for $t$ is given by comparing successive
iterates of the subspace iteration. The value
$\xi_k$ (or $\xi_1$), for small $k$, can be easily computed, since it is the largest in modulus eigenvalue of the $k\times k$ matrix $V^*\mathcal{H} V$, an alternative approach is to approximate it using power methods or some steps of the Arnoldi iteration. 

Once an approximation of $t$ has been computed, if we choose $s$ such that $(1+s)\xi_1>\xi_{k+1}$, then all the eigenvalues  $(1+s)\xi_1,\ldots (1+s)\xi_k$ become larger in modulus
than $\delta$.


\subsection{Solution of the shifted equation}

A popular algorithm for computing the minimal solution of an M-NARE is the Structured Doubling Algorithm (SDA), which, in the formulation of
\cite{glx}, is a system of rational matrix iterations defined as
\begin{equation}\label{eq:SDA}
\begin{aligned}
E_{k+1}=&E_k(I-G_kH_k)^{-1}E_k,\\
F_{k+1}=&F_k(I-H_kG_k)^{-1}F_k,\\
G_{k+1}=&G_k+E_k(I-G_kH_k)^{-1}G_kF_k,\\
H_{k+1}=&H_k+F_k(I-H_kG_k)^{-1}H_kE_k,
\end{aligned}
\end{equation}
with suitable initial values $E_0\in \C^{n\times n}$, $F_0\in
\C^{m\times m}$, $G_0\in \C^{n\times m}$, $H_0\in \C^{m\times n}$.
We say that the SDA is applicable (for a set of initial values $E_0,F_0,G_0,H_0$) if the matrix $I-G_kH_k$ (or equivalently $I-H_kG_k$) is nonsingular for each $k$ otherwise we say that the SDA has a breakdown.

In the case of the M-NARE, choosing the initial values of the SDA  as
\begin{equation}\label{SDAini}
\begin{aligned}
E_0=&I-2\gamma V_{\gamma}^{-1}, &
F_0=&I-2\gamma W_{\gamma}^{-1}, \\
G_0=&2\gamma D_{\gamma}^{-1}CW_{\gamma}^{-1}, &
H_0=&2\gamma W_{\gamma}^{-1}BD_{\gamma}^{-1},\\
A_{\gamma}=&A+\g I,&
D_\g=&D+\g I,\\
W_{\gamma}=&A_{\gamma}-BD_{\gamma}^{-1}C, &
V_{\gamma}=&D_{\gamma}-CA_{\gamma}^{-1}B,
\end{aligned}
\end{equation}
for $\gamma\ge \gamma_*$, defined in \eqref{eq:gamma1}, yields well defined sequences such that $G_k\to X_*$ and $H_k\to Y_*$ where $X_*$ is the minimal nonnegative solution of the M-NARE, while
$Y_*$ is the minimal nonnegative solution of the {\em dual}
M-NARE: $YBY-YA-DY+C=0$. In the critical cases the convergence is linear, while in the other cases is quadratic with rate 
\begin{equation}\label{eq:conv}
\nu=\gap_{\mathcal C_\gamma}(\mathcal H)=\frac{\abs{\Cg(\lambda_n)}}{ \abs{\Cg(\lambda_{n+1})}}.
\end{equation}
Moreover, the value of $\gamma\ge
\gamma_*$ that yields faster convergence is $\gamma=\gamma_*$
 (see \cite{bimp} and the references therein).

The SDA can be applied with minor modifications to the equation associated with the subspace shifted matrix
\[
\wh{\mathcal H}=:\twotwo {\wh D}{-\wh C}{\wh B}{-\wh A}.
\]
It is enough to start the SDA with $E_0,F_0,G_0,H_0$ obtained using formulae \eqref{SDAini} with the new coefficients $\wh A,\wh B,\wh C,\wh D$ instead of $A,B,C,D$ and with the same $\gamma$.

The new matrix $\wh {\mathcal H}$ might not be an M-matrix so that the applicability should be assumed, in this case one can prove that the method converges to the same limit matrices with a rate which is
\begin{equation}\label{eq:convsh}
	\wh \nu=\frac{\max_{i=1,\ldots,n}|\Cg(\wh \la_i)|}{\min_{j=1,\ldots,m}|\Cg(\wh \la_{n+j})|},
\end{equation}
where $\wh {\la}_1,\ldots,\wh \la_{n+m}$ are the eigenvalues of $\wh {\mathcal H}$. The proof of convergence may be obtained with similar manipulations as the ones of \cite{glx,gim}, so we decided to omit it.

Using the same parameter $\gamma$ for both standard SDA and SDA applied to the shifted equation gives different convergence rates according to \eqref{eq:conv} and \eqref{eq:convsh}. An acceleration is obtained in the shifted case if $\wh \nu<\nu$. On the assumption that the central eigenvalues are near to the origin, there exists $s_{\max}$ such that $\wh \nu<\nu$ for any value of the shift parameter $0<s<s_{\max}$.


\section{Numerical experiments}\label{sec:numa}

We present some numerical examples showing the effectiveness of the subspace shift technique, through the algorithms presented in Section \ref{sec:alg}, in solving close-to-critical nonsingular M-NARE, when the assumptions of Section \ref{sec:sushitech} are fulfilled; that is, when the $k$ eigenvalues corresponding to the central subspaces are nonzero and well separated from the other eigenvalues of $\H$.  We recall that these assumptions can be identified dynamically by the algorithm.

We compare the SDA applied to the original equation and to the shifted one. We report the number of steps required by the SDA to converge and the number of steps of the inverse orthogonal iteration in the subspace shift algorithm. All steps of the SDA and the first step of the inverse orthogonal iteration are the most expensive part of the algorithms, since their asymptotic cost is cubic with respect to the size of the matrices; for instance, for $m=n$, the cost of a step of the SDA is $O(n^3)$ elementary arithmetic operations. 

We estimate  the accuracy of the computed solution $\wt X_*$ by means of the relative error
\[
\mbox{err}=\frac{\norm{\wt X_*-X_*}_F}{\norm{X_*}_F},
\]
where $X_*$, if available, is the exact solution or an approximation of it obtained using a higher precision. Otherwise, we use the relative residual
\[
    \mbox{res}=\frac{\norm{\mathcal{R}(\wt X_*)}_F}{\norm{\wt X_*C\wt X_*+B}_F+\norm{A\wt X_*+\wt X_*D}_F}.
\]
In our experiments the Frobenius norm is used. In the tests we use the value $10^{-15}$ as stopping tolerance, a lower tolerance has been used sometimes to monitor the error.

\begin{Test}\label{test:0} \rm
As a first test, we consider the close-to-critical cases of the
transport problem treated in \cite{guolaub, juanglin,bip}. It is
an M-NARE with square coefficients of size $n$ depending on
two parameters $0\le \alpha<1$ and $0<c\le 1$ (for the complete
definition and the meaning of the parameters see \cite{juanglin}).
The problem is critical for $(\alpha,c)=(0,1)$, and it is
close-to-critical if $\alpha$ and $c$ approach simultaneously $0$
and $1$, respectively. In this problem $k=2$ and the two eigenvalues $\la_n$ and $\la_{n+1}$ are well separated from the others.

For $n=4$ and certain values of $\beta$ such that $(\alpha,c)=(\beta,1-\beta)$, we compute the absolute gap of $\H$ ($\gap(\H)=|\la_n-\la_{n+1}|$), the relative sep of its stable subspace ($\mathrm{rsep}(\mathcal W)$, defined in Section \ref{sec:subsep}), the Cayley-transformed gap ($\gap_{\Cg}(\H)$, defined in \eqref{convspeed}) and the same quantities for the shifted matrix $\wh \H$, where we have chosen $\gamma=\gamma_*$ of \eqref{eq:gamma1} in both cases and $s=\xi_3/\xi_1-1$, where $\xi_i$ are the eigenvalues of $\H$ ordered by nondecreasing modulus. We have computed moreover the relative sep of the central subspace, indicated by $\mathrm{rsep}(\mathcal U)$.
The results are reported in Table \ref{tab:test0}.

Since the conditioning of an invariant subspace is proportional to the reciprocal of the relative sep, we observe that the central invariant subspace is much better conditioned than the stable subspace and that the conditioning of the shifted problem is not worse than the one of the original problem.

Recall that $\gap_{\Cg}(\mathcal H)$ and $\gap_{\Cg}(\wh {\mathcal H})$ are the parameters of quadratic convergence of the SDA in the nonshifted and the shifted case, respectively. If $\gap_{\Cg}(\cdot)$ is near to 1, we expect a large number of steps of SDA to obtain the desired accuracy. This suggest that the SDA applied to the shifted equation converge much faster, as shown in the next tests.

\begin{table}\label{tab:test0}
\begin{center}
\begin{tabular}{c|ccc|c|ccc}
$\beta$ & $\gap({\H})$ & $\mathrm{rsep}(\mathcal W)$ & $\gap_{\Cg}(\H)$ & $\mathrm{rsep}(\mathcal U)$  &  $\gap({\wh \H})$ & $\mathrm{rsep}(\wh{\mathcal W})$  & $\gap_{\Cg}(\wh {\H})$ \\ \hline %
$10^{-3}$ & %
$0.11$ & $4.5\cdot 10^{-3}$ & $0.98$ & %
$2.9\cdot 10^{-2}$ & %
$2.5$ & $2.3\cdot 10^{-2}$ & $0.69$ 
\\%
$10^{-6}$ & %
$3.5\cdot 10^{-3}$ & $1.4\cdot 10^{-4}$ & $0.9995$ & %
$2.9\cdot 10^{-2}$ & %
$2.5$ & $7.1\cdot 10^{-4}$ & $0.69$ 
\\%
$10^{-12}$ & %
$3.5\cdot 10^{-6}$ & $1.4\cdot 10^{-7}$ & $\approx 1$ & %
$2.9\cdot 10^{-2}$ & %
$2.5$ & $7.1\cdot 10^{-7}$ & $0.69$ %
\end{tabular}
\caption{Several separation measures for the transport problem with $n=4$}
\end{center}
\end{table}

\end{Test}

\begin{Test}\label{test:1}\rm
We consider the transport problem of Test \ref{test:0}, to which the subspace shift algorithm is applied, where the Riccati equations are solved with the SDA. 

In Table \ref{tab:test1} we give the number of SDA iterations needed to get the best
relative residual for different matrix sizes $n$ and choices of the
parameters $\beta$ (in a stand-alone implementation the number of iterations may be slightly larger due to a non optimal stopping criterion). We provide in parentheses the number of orthogonal iterations needed to approximate the central invariant subspace in the shifted case, where the shift parameter is chosen with the approximation strategy of Section \ref{sec:selshift}.

As
$\beta$ approaches zero, the problem becomes close-to-critical; in
fact $\beta$ is strictly related to the relative gap. 
The table reports also $\gap({\mathcal H})$ and the minimum distance $\delta$ from the two  eigenvalues $\la_n$ and $\la_{n+1}$ to the other eigenvalues of $\H$. 

As one can see the problem
is well suited to be solved by our algorithms since the central
eigenvalues are well separated from the others and $\delta$ is always large enough while the $\gap$ goes to zero; this shows that this example fits adequately the assumptions of Section \ref{sec:sushitech}.

\end{Test}

\begin{table}
\begin{center}
\begin{tabular}{cccc|cc|cc}
$n$ & $\beta$ & $\gap({\mathcal H})$ & $\delta$ & SDA its & SDA res & Alg~1 its & Alg~1 res\\
\hline 32 & $10^{-3}$ & $0.11$ & 0.96 &
 15 & $8.8\cdot 10^{-15}$ & $11\ (12)$ & $4.2\cdot 10^{-16}$\\
32 & $10^{-6}$ & $3.5\cdot 10^{-3}$ & 1.0 &
 20 & $1.0\cdot 10^{-14}$ & $11\ (6)$ & $1.1\cdot 10^{-16}$\\
32 & $10^{-12}$ & $3.5\cdot 10^{-6}$ & 1.0 &
 27 & $8.1\cdot 10^{-15}$ & $11\ (3)$ & $1.1\cdot 10^{-16}$\\
128 & $10^{-3}$ & $0.11$ & 0.95 &
 17 & $1.2\cdot 10^{-13}$ & $13\ (12)$ & $7.7\cdot 10^{-15}$\\
128 & $10^{-6}$ & $3.5\cdot 10^{-3}$ & 1.0 &
 21 & $8.0\cdot 10^{-13}$ & $13\ (6)$ & $3.6\cdot 10^{-16}$\\
128 & $10^{-12}$ & $3.5\cdot 10^{-6}$ & 1.0 &
 30 & $1.5\cdot 10^{-13}$ & $12\ (4)$ & $2.7\cdot 10^{-16}$\\
\end{tabular}
\caption{Number of iterations for Algorithm~1 vs. SDA on the
transport problem}\label{tab:test1}
\end{center}
\end{table}

\begin{Test}\label{test:2} \rm

The second example is taken from \cite{guo01}. The matrix $\mathcal M$ is a random M-matrix of size $n$ and depending on a parameter $\alpha$. As $\alpha$ tends to 0, the matrix $\mathcal M$ tends to a singular matrix. The problem is not close-to-critical, however, there are two central eigenvalues well separated from the others, so the subspace shift algorithm works fine. 

In Table \ref{tab:test2} we report the results for this problem.
The effectiveness of the subspace shift algorithm suggests the possibility to use it in particular problems where a fistful of small eigenvalues are well separated from the others.

\begin{table}
\begin{center}
\begin{tabular}{cccc|cc|cc}
$n$ & $\alpha$ & $\gap({\mathcal H})$ & $\delta$ & SDA its & SDA res & Alg~1 its & Alg~1 res\\
\hline %
50 & $10^{-3}$ & $0.43$ & 43 &
 12 & $3.9\cdot 10^{-16}$ & $4\ (7)$ & $1.1\cdot 10^{-16}$\\
100 & $10^{-3}$ & $0.61$ & 90 &
 13 & $6.1\cdot 10^{-16}$ & $4\ (7)$ & $1.9\cdot 10^{-16}$\\
32 & $10^{-12}$ & $1.1$ & 185 &
 13 & $8.6\cdot 10^{-16}$ & $4\ (7)$ & $3.3\cdot 10^{-16}$\\
\end{tabular}
\caption{Number of iterations for Algorithm~1 vs. SDA on the problem of Test \ref{test:2}}\label{tab:test2}
\end{center}
\end{table}

\end{Test}

\begin{Test}\label{test:3} \rm

The residual is not always a good measure of the accuracy of the solution of a matrix equation. To test the accuracy of the subspace shift algorithm we consider the problems of Test \ref{test:0} and Test \ref{test:2} and compute the solution with double precision to get a solution $X_*$ exact up to 8 significant digits.
Then we run the customary SDA and the SuShi algorithm with precision $10^{-8}$ (single precision). 

The relative error is essentially the same in both cases. For instance, for Test \ref{test:0} with $n=4$ and $\alpha=10^{-3}$ we get for both errors $1.8\cdot 10^{-7}$, for Test \ref{test:2} with $n=100$ and $\alpha=10^{-3}$ we get for both errors $1.4\cdot 10^{-7}$.

\end{Test}


\section{Conclusions}\label{sec:conc}
We have provided a generalization of the shift technique which is aimed to handle close-to-critical nonsymmetric algebraic Riccati equations. The technique consists in computing explicitly the (hopefully moderate-sized) invariant subspace relative to the smallest eigenvalues, which are responsible for the slow convergence of the solution algorithms, and modifying the problem in order to remove them. A theoretical analysis is outlined, not only in terms of eigenvalue location, but also using the more powerful separation metric, which is the one related to the conditioning of the problem; numerical experiments are presented and prove that the application of the shift technique is effective on the analyzed problems.


\section*{Acknowledgments}
We are thankful to an anonymous referee for his/her observations that helped us direct our efforts during the completion of this article.
Some of the results presented here have been obtained while the second author, F.~Poloni, was supported by a doctoral grant of Scuola Normale Superiore, Pisa, Italy.

\bibliographystyle{abbrv}
\bibliography{wcs}
\end{document}